\documentclass[11pt]{article}
\usepackage{amsfonts}
\usepackage{amssymb}
\usepackage{fullpage}
\usepackage{xy}
\xyoption{all}

\title{\bf{A remark on left invariant metrics on compact Lie groups}}
\author{Lorenz J. Schwachh\"ofer\footnote{Research supported by the 
Schwerpunktprogramm Differentialgeometrie of the Deutsche
Forschungsgesellschaft}}

\date{May 7, 2007}
\begin{document}
\maketitle

%%%%%%%%%%%%%%%%%%DEFINITIONS%%%%%%%%%%%%%%%%%%%%%%%%%

\newtheorem{thm}{Theorem}[section]
\newtheorem{lem}[thm]{Lemma}
\newtheorem{prop}[thm]{Proposition}
\newtheorem{df}[thm]{Definition}
\newtheorem{cor}[thm]{Corollary}
\newtheorem{rem}[thm]{Remark}
\newtheorem{ex}[thm]{Example}
\newenvironment{proof}{\medskip
\noindent {\bf Proof.}}{\hfill \rule{.5em}{1em}\mbox{}\bigskip}

\def\eps{\varepsilon}

\def\GlR#1{\mbox{\it Gl}(#1,\R)}
\def\glR#1{{\frak gl}(#1,\R)}
\def\GlC#1{\mbox{\it Gl}(#1,\C)}
\def\glC#1{{\frak gl}(#1,\C)}
\def\Gl#1{\mbox{\it Gl}(#1)}

\def\ov{\overline}
\def\ot{\otimes}
\def\und{\underline}
\def\w{\wedge}
\def\ra{\rightarrow}
\def\lra{\longrightarrow}
\def\less{\prec}

%******************************
\newcommand{\al}{\alpha}
\newcommand{\be}{\beta}
\newcommand{\ga}{\gamma}
\newcommand{\la}{\lambda}
\newcommand{\om}{\omega}
\newcommand{\Om}{\Omega}
\renewcommand{\th}{\theta}
\newcommand{\Th}{\Theta}
%\renewcommand{\phi}{\varphi}

%***************************************
\def\pair#1#2{Q\left(#1,#2\right)}
\def\big#1{\displaystyle{#1}}

%***************************************
\def\N{{\Bbb N}}
\def\Z{{\Bbb Z}}
\def\R{{\Bbb R}}
\def\C{{\Bbb C}}
\def\CP{{\Bbb C} {\Bbb P}}\def\P{{\Bbb P}}
\def\Q{{\Bbb Q}}

\def\HP{{\Bbb H} {\Bbb P}}

\def\O{{\cal O}}
\def\H{{\cal H}}
\def\V{{\cal V}}
\def\F{{\cal F}}

\def\so{{\frak {so}}}
\def\co{{\frak {co}}}
\def\su{{\frak {su}}}
\def\uu{{\frak {u}}}
\def\sl{{\frak {sl}}}
\def\sp{{\frak {sp}}}
\def\csp{{\frak {csp}}}
\def\spin{{\frak {spin}}}
\def\g{{\frak g}}
\def\h{{\frak h}}
\def\k{{\frak k}}
\def\m{{\frak m}}
\def\n{{\frak n}}
\def\t{{\frak t}}
\def\s{{\frak s}}
\def\z{{\frak z}}
\def\p{{\frak p}}
\def\L{{\frak L}}
\renewcommand{\l}{{\frak l}}
\def\X{{\frak X}}
\def\gl{{\frak {gl}}}
\def\hol{{\frak {hol}}}
\renewcommand{\frak}{\mathfrak}
\renewcommand{\Bbb}{\mathbb}

\def\be{\begin{equation}}
\def\ee{\end{equation}}
\def\bi{\begin{enumerate}}
\def\ei{\end{enumerate}}
\def\ba{\begin{array}}
\def\ea{\end{array}}
\def\bea{\begin{eqnarray}}
\def\eea{\end{eqnarray}}
\def\ben{\begin{enumerate}}
\def\een{\end{enumerate}}

\def\dq{\slash \!\!\!\! \slash}
\def\codim{\mbox{\rm codim}}

%\begin{abstract}
%\noindent
%\end{abstract}

%%%%%%%%%%%%%%%%%%%%%%%%%%%%%%%%%%%%%%%%%%%%%%%%%%%%%%%%%%%%%%%%%%%%%%%
\section{Introduction}
%%%%%%%%%%%%%%%%%%%%%%%%%%%%%%%%%%%%%%%%%%%%%%%%%%%%%%%%%%%%%%%%%%%%%%%

The investigation of manifolds with non-negative sectional curvature is one of the classical fields of study in global Riemannian geometry. While there are few known obstruction for a closed manifold to admit metrics of non-negative sectional curvature, there are relatively few known examples and general construction methods of such manifolds (see \cite{Z} for a detailed survey).

In this context, it is particularly interesting to investigate left invariant metrics on a compact connected Lie group $G$ with Lie algebra $\g$. These metrics are obtained by left translation of an inner product on $\g$. If this metric is  biinvariant then its sectional curvature is non-negative, and it is known that the set of inner products on $\g$ whose corresponding left invariant metric on $G$ has non-negative sectional curvature is a connected cone; indeed, each such inner product can be connected to a biinvariant one by a canonical path (\cite{T}).

In the present article, it is shown that the stretching of the biinvariant metric in the direction of a subalgebra of $\g$ almost always produces some negative sectional curvature of the corresponding left invariant metric on $G$. In fact, the following theorem answers a question raised in \cite[Problem 1, p.9]{Z}.

\begin{thm} \label{main}
Let $H \subset G$ be compact Lie groups with Lie algebras $\h \subset \g$, let $Q$ be a biinvariant inner product on $\g$, and for $t > 0$ let $g_t$ be the left invariant metric on $G$ induced by the inner product
\be \label{def:Q_t}
Q_t := t\ Q|_\h + Q|_{\h^\perp}.
\ee
If there is a $t > 1$ such that $g_t$ has non-negative sectional curvature, then then the semi-simple part of $\h$ is an ideal of $\g$.
\end{thm}

Note that this condition is indeed optimal: if $t \leq 1$ then $g_t$ is known to have non-negative sectional curvature, and if the semi-simple part of $\h$ is an ideal of $\g$ then $g_t$ has non-negative sectional curvature even for $t \leq 4/3$ (\cite{GZ1}).

There is yet another reason why this result is of interest. One of the most spectacular source of examples of manifolds of non-negative sectional curvature of the last decade was given in \cite{GZ1} where it was shown that any closed cohomogeneity one manifold whose non-principal orbits have codimension at most two admit invariant metrics of non-negative sectional curvature. Their construction is based on glueing homogeneous disk bundles of rank $\leq 2$ along a totally geodesic boundary which is equipped with a normal homogeneous metric. 

The reason for this construction to work is due to the fact that the structure group of the fibers is contained in $H = SO(k)$ where $k$ is the rank of the bundle. If $k \leq 2$, then $H$ is abelian, so that the metrics $g_t$ from Theorem \ref{main} have non-negative sectional curvature for some $t > 1$. 

Our result now suggests that for most subgroups $H' \subset H$, the metric on $G/H'$ induced by the metric $g_t$ with $t  > 1$ from Theorem \ref{main} will have some negative sectional curvature as well. Therefore, it will be difficult to find more examples of non-negatively curved metrics on homogeneous vector bundles over $G/H$ with normal homogeneous collar. Also, note that there are examples of cohomogeneity one manifolds, including the Kervaire spheres, which do not admit invariant metric of non-negative sectional curvature at all (\cite{GVWZ}).

%It is a pleasure to thank Wolfgang Ziller for helpful remarks and suggestions on this work.

%%%%%%%%%%%%%%%%%%%%%%%%%%%%%%%%%%%%%%%%%%%%%%%%%%%%%%%%%%%%%%%%%%%%%%%
\section{Proof of Theorem \ref{main}}
%%%%%%%%%%%%%%%%%%%%%%%%%%%%%%%%%%%%%%%%%%%%%%%%%%%%%%%%%%%%%%%%%%%%%%%

Let $H \subset G$, $\h \subset \g$, $Q_t$ and $g_t$ be as in Theorem \ref{main}, and let $\m := \h^\perp$, so that we have the orthogonal splitting
\be \label{eq:g=h+m}
\g = \h \oplus \m.
\ee

\noindent Then a calculation shows that for any $s > 0$ and $t := s/(1+s)$, the multiplication map 
\be
(H \times G,  s Q|_\h + Q|_\g) \longrightarrow (G, g_t)
\ee
becomes a Riemannian submersion (cf. e.g. \cite{Ch}). But $s Q|_\h + Q|_\g$ is a biinvariant metric on $H \times G$ which therefore has non-negative sectional curvature, and by O'Neill's formula so does $Q_t$. Since $Q_1 = Q$ is a biinvariant metric, and any $t \in (0,1)$ can be written as $t = s/(1+s)$ for some $s > 0$, we conclude that $Q_t$ has non-negative sectional curvature for all $t \leq 1$. 

We shall divide the proof of Theorem \ref{main} into two lemmas.

\begin{lem} \label{lem:criterion}
Suppose that the metric $Q_t$ on $G$ has non-negative sectional curvature for some $t > 1$. Then for all $x, y \in \g$ with $[x, y] = 0$ we must have $[x_\h, y_\h] = 0$, where $x = x_\h + x_\m$ and $y = y_\h + y_\m$ is the decomposition according to (\ref{eq:g=h+m}).
\end{lem}

\begin{proof}
The curvature tensor $R^t$ of the metric $g_t$ has been calculated e.g. in \cite{GZ1}. Namely, for elements $x = x_\h + x_\m$ and $y = y_\h + y_\m$ we have
\be \label{eq:curvature} \ba{llll}
Q_t(R^t(x, y)y, x) & = & & \frac14 ||\ [x_\m, y_\m]_\m + t [x_\h, y_\m] + t [x_\m, y_\h]\ ||_Q^2\\
& & + & \frac14 t ||\ [x_\h, y_\h]\ ||_Q^2 + \frac12 t (3 - 2t) Q([x_\h, y_\h], [x_\m, y_\m]_\h) + (1 - \frac34 t) 
||\ Ê[x_\m, y_\m]_\h\ ||_Q^2.
\ea \ee

Let $x^t := t x_\h + x_\m$ and $y^t := t y_\h + y_m$. Then, using that $[\h, \m] \subset \m$, it follows that
\[
[x^t, y^t]_\h = t^2 [x_\h, y_\h] + [x_\m, y_\m]_\h \mbox{\hspace{1cm} and \hspace{1cm}} [x^t, y^t]_\m =  [x_\m, y_\m]_\m + t [x_\h, y_\m] + t [x_\m, y_\h].
\]
If we assume that $[x^t, y^t] = 0$, then $[x_\m, y_\m]_\h = - t^2 [x_\h, y_\h]$ and $[x_\m, y_\m]_\m + t [x_\h, y_\m] + t [x_\m, y_\h] = 0$. Substituting this into (\ref{eq:curvature}) yields
\be \label{eq:curvature2} \ba{lll}
Q_t(R^t(x, y)y, x) & = & \left( \frac14 t - \frac12 t^3 (3 - 2t) + (1 - \frac34 t) t^4\right) ||\ [x_\h, y_\h]\ ||_Q^2\\ \\
& = & - \frac14 t (t - 1)^3 (1 + 3 t) ||\ [x_\h, y_\h]\ ||_Q^2.
\ea \ee

If this expression is non-negative for some $t > 1$, then $[x_\h, y_\h] = 0$. Thus, $[x^t_\h, y^t_\h] = t^2 [x_\h, y_\h] = 0$ whenever $[x^t, y^t] = 0$.
\end{proof}

\begin{lem} Let $\h \subset \g$ be a Lie subalgebra such that all $x, y \in \g$ with $[x, y] = 0$ satisfy $[x_\h, y_\h] = 0$, where $x = x_\h + x_\m$ and $y = y_\h + y_\m$ is the decomposition according to (\ref{eq:g=h+m}). Then the semi-simple part of $\h$ is an ideal of $\g$.
\end{lem}

\begin{proof} Let $\h = \z(\h) \oplus \h_1 \oplus \ldots \oplus \h_r$ be the decomposition into the center and simple ideals. Then $[x_\h, y_\h] = 0$ iff $[x_{\h_k}, y_{\h_k}] = 0$ for all $k$. Also, the semi-simple part of $\h$ is an ideal of $\g$ iff $\h_k \lhd \g$ for all $k$. Thus, it sufices to show the lemma for all $\h_k$, whence we shall assume for the rest of the proof that $\h$ is {\em simple}.

\

\noindent {\em Step 1. Let $y \in \m$ be such that there is an $0 \neq x \in \h$ with $[x,y] = 0$. Then $[\h, y] = 0$.}

\

For any $a \in \m$ and $t \in \R$, we have $[Ad_{\exp(ta)} x, Ad_{\exp(ta)} y] = Ad_{\exp(ta)} [x, y] = 0$, hence by hypothesis $[(Ad_{\exp(ta)} x)_\h, (Ad_{\exp(ta)} y)_\h] = 0$.

But $[a, x] \in [\m, \h] \subset \m$, hence $(Ad_{\exp(ta)} x)_\h = x + O(t^2)$, whereas $(Ad_{\exp(ta)} y)_\h = t [a,y]_\h  + \frac12 t^2 [a, [a, y]]_\h + O(t^3)$. Therefore, for all $t \in \R$ we have
\be \label{eq:powerseries} \ba{lll}
0 & = & \left[(Ad_{\exp(ta)} x)_\h, (Ad_{\exp(ta)} y)_\h\right] = t [x, [a,y]_\h] + \frac12 t^2 [x, [a, [a, y]]_\h] + O(t^3)\\ \\
& = & t [x, [a,y]]_\h + \frac12 t^2 [x, [a, [a, y]]]_\h + O(t^3).
\ea
\ee
The last equation follows since for all $x \in \h$ and $z = z_\h + z_\m$ we have $[x, z_\h] \in \h$ and $[x, z_\m] \in \m$, whence $[x, z_\h] = [x,z]_\h$. Thus, we must have $[x, [a,y]]_\h = 0$ for all $a \in \m$. On the other hand, if $a \in \h$ then $[x, [a,y]] \in [\h, [\h, \m]] \subset \m$, hence $[x, [a,y]]_\h = 0$ for all $a \in \h$ as well, and therefore,

\be \label{eq:yxh=0} \ba{lll}
0 = Q([x, [\g,y]], \h) = Q(\g, [[x, \h], y]), & \mbox{i.e.,} & [[x, \h], y] = 0.
\ea \ee

By \cite[Lemma 4.4]{S} and the simplicity of $\h$, it follows that $\h$ is the linear span of $x$, $[x, \h]$ and $[[x, \h], [x, \h]]$. Since $[x,y] = 0$, and $[[x, \h], y]  = 0$ by (\ref{eq:yxh=0}), this together with the Jacobi identity now implies that $[\h, y] = 0$ as claimed.

\

\noindent {\em Step 2. Let $y \in \m$ be such that $[\h, y] = 0$. Let $\g' \lhd \g$ and $\g'' \lhd \g$ be the ideals generated by $\h$ and $y$, respectively. Then $Q(\g', \g'') = 0$ and $[\g', \g''] = 0$. In particular, $Q(\g', y) = 0$}

\

First, note the it suffices to show that $Q(\h, \g'') = 0$. For if this is the case, it then follows that $Q(ad(\g)^n(\h), \g'') = Q(\h, ad(\g)^n(\g'')) = Q(\h, \g'') = 0$, which implies that $Q(\g', \g'') = 0$. Hence, $Q([\g', \g''], \g) = Q(\g',[\g'', \g]) = Q(\g', \g'') = 0$ so that $[\g', \g''] = 0$ follows.

By \cite[Lemma 4.4]{S}, $\g''$ is the linear span of $y$, $[\g, y]$ and $[\g, [\g, y]]$. Since $y \in \m$, we have $Q(y, \h) = 0$, and $Q([\g, y], \h) = Q(\g, [\h, y]) = 0$ by hypothesis. Thus, $Q(\h, \g'') = 0$ will be demonstrated once we show that $Q([\g, [\g, y]], \h) = 0$.

For a fixed $h \in \h$, we define the bilinear form $\al_h$ on $\g$ by
\[
\al_h(a,b) := Q([a, [b, y]], h).
\]

Thus, our goal shall be to show that $\al_h = 0$ for all $h \in \h$. Note that $\al_h(a,b) - \al_h(b,a) = Q([a, [b, y]] - [b, [a, y]], h) = Q([[a,b], y], h) = -Q([a,b], [h, y]) = 0$ by hypothesis, hence $\al_h$ is {\em symmetric}. If $b \in \h$, then $[b, y] = 0$ by hypothesis, so that $\al_h(\g,\h) = 0$.

By or hypothesis and step 1, (\ref{eq:powerseries}) holds for {\em all} $x \in \h$, thus the vanishing of the $t^2$-coefficient of (\ref{eq:powerseries}) implies that 
\[
0 = Q([\h, [a, [a, y]], \h) = Q([a, [a, y]], [\h, \h]) = Q([a, [a, y]], \h)  \mbox{ for all $a \in \m$}.
\]
Thus, $\al_h(a,a) = 0$ for all $a \in \m$ and therefore, $\al_h = 0$ for all $h \in \h$ as asserted.

\

\noindent {\em Step 3. $\h \lhd \g$.}

\

Let $\g' \lhd \g$ be the ideal generated by $\h$. By steps 1 and two, it follows that there cannot be an $0 \neq x \in \h$ and $0 \neq y \in \m \cap \g'$ with $[x, y] = 0$. This immediately implies that $rk(\h) = rk(\g')$.

If $rk(\h) = rk(\g') = 1$, then $\h = \g' \lhd \g$ and we are done. If $rk(\h) \geq 2$ then we can choose linearly independent elements $x_1, x_2 \in \h$ with $[x_1, x_2] = 0$. If $\m \cap \g' \neq 0$, then the restrictions of $ad_{x_i}$ to $\m \cap \g'$ have common eigenspaces, i.e., there is an orthogonal decomposition 
\[
\m \cap \g' = V_1 \oplus \ldots \oplus V_m
\]
into two-dimensional subspaces $V_k$ on which both $ad_{x_i}$ act by a multiple of rotation by a right angle. Therefore, for each $k$, there is a suitable $0 \neq x^k \in span(x_1, x_2) \subset \h$ such that $[x^k, V_k] = 0$ which is a contradiction. Therefore, $\m \cap \g' = 0$, i.e., $\h = \g' \lhd \g$.
\end{proof}

%%%%%%%%%%%%%%%%%%%%%%%%%%%%%%%%%%%%%%%%%%%%%%%%%%%%%%%%%%%%%%%%%%%%%%%

{\sc
\noindent
Fachbereich Mathematik, Universit\"at Dortmund, 44221 Dortmund, Germany
\

\noindent
Email:\, 
{\tt lschwach@math.uni-dortmund.de}
}

\end{document}